\documentclass[10pt]{article}
   \usepackage{amsmath,amsfonts,amsthm,amssymb,amscd,graphicx}
\theoremstyle{plain}
\newtheorem{thm}{Theorem}[section]
\newtheorem{prop}[thm]{Proposition}
\newtheorem{lemma}[thm]{Lemma}
\newtheorem{cor}[thm]{Corollary}

\theoremstyle{definition}
\newtheorem{defn}[thm]{Definition}

\theoremstyle{remark}
\newtheorem{rmk}[thm]{Remark}

\newcommand{\vol}{\ensuremath{\mathsf{vol}}}
\newcommand{\hol}{\ensuremath{\mathrm{Hol}_g (M)}}

\newcommand{\G}{\ensuremath{\operatorname{G_2}}}
\newcommand{\SP}{\ensuremath{\operatorname{Spin}(7)}}

\newcommand{\Gs}{\ensuremath{\operatorname{G_2}}{-structure}}
\newcommand{\SPs}{\ensuremath{\operatorname{Spin}(7)}{-structure}}
\newcommand{\R}{\ensuremath{\mathbb R}}

\newcommand{\Ph}{\ensuremath{\Phi}}

\newcommand{\wone}{\ensuremath{\Omega^1_8}}
\newcommand{\wtws}{\ensuremath{\Omega^2_7}}
\newcommand{\wtwt}{\ensuremath{\Omega^2_{21}}}
\newcommand{\wthe}{\ensuremath{\Omega^3_8}}
\newcommand{\wthf}{\ensuremath{\Omega^3_{48}}}
\newcommand{\wfoo}{\ensuremath{\Omega^4_1}}
\newcommand{\wfos}{\ensuremath{\Omega^4_7}}
\newcommand{\wfot}{\ensuremath{\Omega^4_{27}}}
\newcommand{\wfoth}{\ensuremath{\Omega^4_{35}}}

\newcommand{\st}{\ensuremath{\ast}}
\newcommand{\hk}{\mathbin{\! \hbox{\vrule height0.3pt width5pt depth 0.2pt \vrule height5pt width0.4pt depth 0.2pt}}}

\newcommand{\lies}{\ensuremath{\operatorname{\mathfrak {so}(7)}}}
\newcommand{\tr}{\ensuremath{\operatorname{Tr}}}

\newcommand{\ddx}[1]{\ensuremath{\frac{\del}{\del\:\!\! x^{#1}}}}
\newcommand{\dx}[1]{\ensuremath{d\:\!\! x^{#1}}}
\newcommand{\nab}[1]{\ensuremath{\nabla_{\! \! #1 \,}}}
\newcommand{\del}{\ensuremath{\partial}}

\newcommand{\ddt}{\ensuremath{\frac{\partial}{\partial t}}}

\newcommand{\dl}{\ensuremath{\delta}}
\newcommand{\adm}{\ensuremath{\mathcal{A}}}

\numberwithin{equation}{section}
\numberwithin{table}{section}
\numberwithin{figure}{section}

\begin{document}

\title{Flows of \SP-structures}
\author{Spiro Karigiannis \\ {\em Mathematical Institute, University of Oxford} \\ {\tt karigiannis@maths.ox.ac.uk} }

\maketitle

\begin{abstract}
We consider flows of \SPs s.  We use local coordinates to describe
the torsion tensor of a \SPs\ and derive the evolution equations
for a general flow of a \SPs\ $\Ph$ on an $8$-manifold $M$.
Specifically, we compute the evolution of the metric and the
torsion tensor. We also give an explicit description of the decomposition
of the space of forms on a manifold with \SPs, and derive an analogue
of the second Bianchi identity in \SP-geometry. This identity
yields an explicit formula for the Ricci tensor and
part of the Riemann curvature tensor in terms of the torsion.
\end{abstract}

\section{Introduction} \label{introsec}

This paper discusses general flows of \SPs s in a manner
similar to the author's analogous results for flows of \Gs s, which
were studied in~\cite{K3}. Many of the calculations are similar in
spirit, although more involved, so we often omit proofs. The reader is
advised to familiarize themselves with~\cite{K3} first.

A general evolution of a \SPs\ is described by a symmetric tensor
$h$ and a skew-symmetric tensor $X$ satisfying some further algebraic
condition, and it is only $h$ which affects the evolution of the
associated Riemannian metric. However, the evolution of the
torsion tensor is determined by both $h$ and $X$.

In Section~\ref{SPstructuressec}, we review \SPs s, the
decomposition of the space of forms, and the torsion tensor
of a \SPs. In Section~\ref{generalflowssec} we compute the evolution
equations for the metric and the torsion tensor for a general flow of
\SPs s. In Section~\ref{bianchisec}, we apply our evolution equations
to derive a Bianchi-type identity in \SP-geometry.
This leads to an explicit formula for the Ricci tensor of a general \SPs\
in terms of the torsion. An Appendix collects various identities in
\SP-geometry.

The notation used in this paper is identical to that of~\cite{K3}.
Throughout this paper, $M$ is a (not necessarily compact) smooth
manifold of dimension $8$ which admits a \SPs.

{\bf Acknowledgements.}  This article is based on a talk given by the
author at the 10th International Conference on Differential Geometry
and its Applications, in honour of the $300^{\mathrm{th}}$ anniversary
of the birth of Leonhard Euler, Czech Republic.

\section{Manifolds with \SPs} \label{SPstructuressec}

In this section we review the concept of a \SPs\ on a manifold $M$ and
the associated decompositions of the space of forms. More details
about \SPs s can be found, for example, in~\cite{BS, J1, J4, K1}. We
also describe explicitly the torsion tensor associated to a \SPs.

Consider an $8$-manifold $M$ with a \SP\ structure
$\Ph$. The existence of such a structure is a {\it topological condition}.
The space of $4$-forms $\Ph$ on $M$ which determine a \SPs\ is a
subbundle $\adm$ of the bundle $\Omega^4$ of $4$-forms on $M$,
called the bundle of {\it admissible\/} $4$-forms. This is {\it not\/}
a vector subbundle, and unlike the \G\ case, it is not even an
open subbundle.

A \SPs\ $\Ph$ determines a Riemannian metric $g_{\Ph}$ and an
orientation in a non-linear fashion which we now describe. Let $v$ be
a non-zero tangent vector at a point $p$, and extend to $(v, e_1,
\ldots, e_7)$ a local frame near $p$. We define
\begin{eqnarray*}
B_{ij}(v) & = & \left( (e_i \hk v \hk \Ph)
\wedge (e_j \hk v \hk \Ph) \wedge (v \hk \Ph) \right)
(e_1, \ldots, e_7) \\ A(v) & = & \left( (v \hk \Ph) \wedge \Ph \right)
(e_1, \ldots, e_7)
\end{eqnarray*}
Then the metric $g_{\Ph}$ is defined by
\begin{equation} \label{gijeq}
{\left( g_{\Ph}(v,v) \right)}^2 = -\frac{7^3}{6^{\frac{7}{3}}} \,
\frac{(\det B_{ij}(v))^{\frac{1}{3}}}{A(v)^3}
\end{equation}
More details can be found in~\cite{K1}, although that paper uses
a different orientation convention (see also~\cite{K2}.) We will not
have need for this explicit formula.

The metric $g_{\Ph}$ and orientation (determined by the volume form)
determine a Hodge star operator $\st$, and the $4$-form $\Ph$ is {\it
self-dual}. That is, $\st \Ph = \Ph$. The metric also determines the
Levi-Civita connection $\nabla$, and the manifold $(M, \Ph)$ is called
a \SP\ manifold if $\nab{} \Ph = 0$. This is a nonlinear partial
differential equation for $\Ph$, since $\nabla$ depends on $g$ which
depends non-linearly on $\Ph$. Such manifolds (where $\Ph$ is
parallel) have Riemannian holonomy $\hol$ contained in the group $\SP
\subset \mathrm{SO}(8)$. A parallel \SPs\ is also called
{\it torsion-free}.

\subsection{Decomposition of the space of forms} \label{formssec}

The existence of a \SPs\ $\Ph$ on $M$ (with no condition on $\nab{}
\Ph$) determines a decomposition of the spaces of differential
forms on $M$ into irreducible \SP\ representations.
We will see explicitly that the spaces $\Omega^2$, $\Omega^3$, and 
$\Omega^4$ decompose as
\begin{align*}
& \Omega^2 = \wtws \oplus \wtwt \qquad \qquad \Omega^3 = \wthe \oplus \wthf \\
& \qquad \Omega^4 = \wfoo \oplus \wfos \oplus \wfot \oplus \wfoth
\end{align*}
where $\Omega^k_l$ has (pointwise) dimension $l$ and this
decomposition is orthogonal with respect to the metric $g$.
For $k=2$ and $k=3$, the explicit descriptions are as follows:
\begin{align}
\label{wtwdesc} & \wtws \, = \, \{ \beta \in \Omega^2; \, \st ( \Ph
\wedge \beta ) = -3 \beta \} \quad \quad \wtwt \, = \, \{ \beta
\in \Omega^2; \, \st ( \Ph \wedge \beta) = \beta \} \\
\label{wthdesc} & \wthe \, = \, \{ X \hk \Ph;\, X \in \Gamma(TM) \}
\qquad \qquad \, \, \wthf \, = \, \{ \gamma \in \Omega^3;
\, \gamma \wedge \Ph = 0 \}
\end{align}
For $k>4$, we have $\Omega^k_l = \st \Omega^{8-k}_l$.

We need these decompositions in local coordinates. The following
proposition is easy to verify.

\begin{prop} \label{w23prop}
Let $\beta_{ij}$ be a $2$-form, $\gamma_{ijk}$ a $3$-form, and $X^k$ a
vector field. Then
\begin{align*}
& \beta_{ij} \in \wtws \, \Leftrightarrow \, \beta_{ab} g^{ap} g^{bq}
\Ph_{pqij} = -6 \, \beta_{ij} & \beta_{ij} \in \wtwt \, \Leftrightarrow \,
\beta_{ab} g^{ap} g^{bq} \Ph_{pqij} = 2 \, \beta_{ij} \\
& \gamma_{ijk} \in \wthe \, \Leftrightarrow \, 
\gamma_{ijk} = X^l \Ph_{ijkl} & \gamma_{ijk} \in \wthf \,
\Leftrightarrow \, \gamma_{ijk} g^{ia} g^{jb} g^{kc} \Ph_{abcd} = 0
\end{align*}
and the projection operators $\pi_7$ and $\pi_{21}$ on $\Omega^2$ are
given by
\begin{eqnarray} \label{wtwsprojeq}
\pi_7 (\beta)_{ij} & = & \frac{1}{4} \, \beta_{ij} - \frac{1}{8} \, 
\beta_{ab} g^{ap} g^{bq} \Ph_{pqij} \\ \label{wtwtprojeq} \pi_{21}
(\beta)_{ij} & = & \frac{3}{4} \, \beta_{ij} + \frac{1}{8} \, \beta_{ab}
g^{ap} g^{bq} \Ph_{pqij}
\end{eqnarray}
\end{prop}

\begin{rmk} \label{liesrmk}
One can show using Proposition~\ref{w23prop} and
Lemma~\ref{SPidentitieslemma} that if $\beta_{ij} \in \wtwt$,
\begin{equation*}
\beta_{ab} g^{bl} \Ph_{lpqr} = \beta_{pi} g^{ij} \Ph_{jqra} +
\beta_{qi} g^{ij} \Ph_{jrpa} + \beta_{ri} g^{ij} \Ph_{jpqa}
\end{equation*}
which can then be used to show that $\wtwt$ is a Lie algebra with
respect to the commutator of matrices:
\begin{equation*}
[ \beta, \mu]_{ij} = \beta_{il} g^{lm} \mu_{mj} - \mu_{il} g^{lm}
\beta_{mj}
\end{equation*}
In fact, $\wtwt \cong \lies$, the Lie algebra of $\SP$.
\end{rmk}

The decomposition of the space $\Omega^4$ of $4$-forms can be
understood by considering the infinitesmal action of $\mathrm{GL}(8,
\mathbb R)$ on $\Ph$. Let $A = A^i_l \in \mathfrak{gl}(8, \mathbb R)$.
Hence $e^{At} \in \mathrm{GL}(8, \mathbb R)$, and we have
\begin{equation*}
e^{A t} \cdot \Ph = \frac{1}{24} \, \Ph_{ijkl} \, (e^{At} \dx{i}) \wedge 
(e^{At} \dx{j}) \wedge (e^{At} \dx{k}) \wedge (e^{At} \dx{l}) 
\end{equation*}
Differentiating with respect to $t$ and setting $t = 0$, we obtain:
\begin{equation*}
\left. \frac{d}{dt} \right|_{t=0} (e^{At} \cdot \Ph) =
A^m_i \, \dx{i} \wedge \left( \ddx{m} \hk \Ph \right)
\end{equation*}
Now let $A^m_i = g^{mj} A_{ij}$, and decompose $A_{ij} = S_{ij} +
C_{ij}$ into symmetric and skew-symmetric parts, where $S_{ij} =
\frac{1}{2} (A_{ij} + A_{ji})$ and $C_{ij} = \frac{1}{2} (A_{ij} -
A_{ji})$. We have a map
\begin{eqnarray*}
D & : &  \mathfrak{gl}(8, \mathbb R) \to \Omega^4 \\
D & : & A \mapsto \left. \frac{d}{dt} \right|_{t=0} (e^{At} \cdot \Ph) =
S_{ij} g^{jm} \, \dx{i} \wedge \left( \ddx{m} \hk \Ph \right) +
C_{ij} g^{jm} \, \dx{i} \wedge \left( \ddx{m} \hk \Ph \right)  
\end{eqnarray*}
\begin{prop} \label{wtwfdescprop} The kernel of $D$ is isomorphic to
the subspace $\wtwt$. It is also isomorphic to the Lie algebra
$\lies$ of the Lie group \SP\ which is the subgroup of
$\mathrm{GL}(8, \mathbb R)$ which preserves $\Ph$.
\end{prop}
\begin{proof}
Since we are defining \SP\ to be the group preserving $\Ph$, the
kernel of $D$ is isomorphic to $\lies$ by definition. To
show explicitly that this is isomorphic to $\wtwt$, suppose that
$C_{ij}$ is in $\wtwt$. Then $D(C)$ is
\begin{equation*}
\frac{1}{24} \, \left( C^m_i \Ph_{mjkl} + C^m_j \Ph_{imkl} +
C^m_k \Ph_{ijml} + C^m_l \Ph_{ijkm} \right) \dx{i} \wedge
\dx{j} \wedge \dx{k} \wedge \dx{l}
\end{equation*}
From Proposition~\ref{w23prop}, we have $C_{ij} = \frac{1}{2}
\, C_{ab} g^{ap} g^{bq} \Ph_{pqij}$. Using this together with the
final equation of Lemma~\ref{SPidentitieslemma}, one can compute that
\begin{multline*}
C^m_i \Ph_{mjkl} + C^m_j \Ph_{imkl} +
C^m_k \Ph_{ijml} + C^m_l \Ph_{ijkm} = \\ -3 \left( 
C^m_i \Ph_{mjkl} + C^m_j \Ph_{imkl} + C^m_k \Ph_{ijml} + C^m_l \Ph_{ijkm}
\right)
\end{multline*}
and hence $D(C)= 0$. Thus $\wtwt$ is in the kernel of $D$.
We show below that $D$ restricted to $\wtws$ or to $S^2(T)$ is
injective. This completes the proof.
\end{proof}

By counting dimensions, we must have $\wfos = D(\wtws)$
and also $\wfoo \oplus \wfoth = D(S^2)$.
We now proceed to establish these explicitly. The proofs of the next two
propositions are very similar to analogous results in~\cite{K3} and are
left to the reader.

\begin{prop} \label{starDAprop}
Suppose that $A_{ij}$ is a tensor. Consider the $4$-form
$D(A)$ given by
\begin{equation*}
D(A) = A_{ij} g^{jm} \, \dx{i} \wedge \left( \ddx{m} \hk \Ph \right)
\end{equation*}
or equivalently
\begin{equation} \label{DAeq}
D(A)_{ijkl} = A_{im} g^{mn} \Ph_{njkl} + A_{jm} g^{mn} \Ph_{inkl}
+ A_{km} g^{mn} \Ph_{ijnl} + A_{lm} g^{mn} \Ph_{ijkn}
\end{equation}
Then the Hodge star of $D(A)$ is
\begin{equation*}
\st D(A) = D(\bar A) = \bar A_{ij} g^{jm} \, \dx{i} \wedge \left( \ddx{m}
\hk \Ph \right)
\end{equation*}
where $\bar A_{ij} = \frac{1}{4} \tr_g(A) g_{ij} - A_{ji}$. That is,
as a matrix, $\bar A = \frac{1}{4} \tr_g(A) I - A^T$.
\end{prop}

\begin{prop} \label{metricDAprop}
Suppose $A_{ij}$ and $B_{ij}$ are two tensors. Let $D(A)$ and
$D(B)$ be their corresponding forms in $\Omega^4$. We write
$A_{ij} = \frac{1}{8} \tr_g(A) g_{ij} + (A_0)_{ij} + (A_7)_{ij}$, where
$A_0$ is the symmetric traceless component of $A$, and $A_7$ is the
component in $\wtws$. Similarly for $B$. (We can assume they have
no $\wtwt$ component since that is in the kernel of $D$.) Then we have
\begin{equation*}
g_{\Ph} (D(A), D(B)) \, = \frac{7}{2} \tr_g(A)
\tr_g(B) + 4 \tr_g(A_0 B_0) - 16 \tr_g(A_7 B_7)
\end{equation*}
where $A_0 B_0$ and $A_7 B_7$ mean matrix multiplication.
\end{prop}

\begin{cor} \label{DAinjectivecor}
The map $D : \mathfrak{gl}(8,\mathbb R) \to \Omega^4$ is injective
on $S^2 \oplus \wtws$. It is therefore an isomorphism onto
its image, $\wfoo \oplus \wfos \oplus \wfoth$.
\end{cor}
\begin{proof}
This follows immediately from Proposition~\ref{metricDAprop}, since
if $D(A) = 0$ and $A$ is pure trace, traceless symmetric, or in $\wtws$,
we see that $A = 0$.
\end{proof}

We still need to understand the space $\wfot$. To do this we give
another characterization of the space of $4$-forms using the \SPs,
which may be well-known to experts but has apparently not
appeared in print before.

\begin{defn} \label{Lambdadefn}
We define a \SP-equivariant linear operator $\Lambda_{\Ph}$ on
$\Omega^4$ as follows. Let $\sigma \in \Omega^4$. Use the
notation $(\sigma \cdot \Ph)_{ijkl}$ to denote
$\sigma_{ijmn} g^{mp} g^{nq} \Ph_{pqkl}$. Then
$\Lambda_{\Ph}(\sigma) \in \Omega^4$ is given by
\begin{equation*}
(\Lambda_{\Ph}(\sigma))_{ijkl} = (\sigma \cdot \Ph)_{ijkl} +
(\sigma \cdot \Ph)_{iklj} + (\sigma \cdot \Ph)_{iljk} +
(\sigma \cdot \Ph)_{jkil} + (\sigma \cdot \Ph)_{jlki} +
(\sigma \cdot \Ph)_{klij}
\end{equation*}
\end{defn}
We now explain the motivation for introducing this operator
$\Lambda_{\Ph}$. If $\sigma \in \wfot$, then $g(\sigma, D(A)) = 0$
for all $A \in \mathfrak{gl}(8, \R)$ since $D(A) \in \wfoo \oplus
\wfos \oplus \wfoth$ and the splitting is orthogonal. Writing this
in coordinates using~\eqref{DAeq} gives
\begin{equation*}
\sigma \in \wfot \qquad \Leftrightarrow \qquad \sigma_{abcd} \Ph_{ijkl}
g^{jb} g^{kc} g^{ld} = 0 \quad \text{for all } a, i = 1 , \ldots, 8
\end{equation*}
Taking the above expression and contracting it with $\Ph$, and using
Lemma~\ref{SPidentitieslemma}, after some laborious calculation one can
show that
\begin{equation*}
\sigma \in \wfot \qquad \Leftrightarrow \qquad \sigma_{ijkl} =
\frac{1}{4} (\Lambda_{\Ph}(\sigma))_{ijkl}
\end{equation*}
which says that $\wfot$ is an eigenspace of $\Lambda_{\Ph}$ with
eigenvalue $+4$. Suppose now that $\sigma = D(A) \in \wfoo \oplus
\wfos \oplus \wfoth$. Then another brute force calculation using
Definition~\ref{Lambdadefn} and Lemma~\ref{SPidentitieslemma} shows
\begin{equation*}
\Lambda_{\Ph}(D(A)) = D(6 A^T - 6 A - 3 \tr_g(A) I)
\end{equation*}
and from the above relation it is a simple matter to verify the
following characterization of $\Omega^4$.
\begin{prop} \label{Lambdaprop}
The spaces $\wfoo$, $\wfos$, $\wfot$, and $\wfoth$ are all eigenspaces
of $\Lambda_{\Ph}$ with distinct eigenvalues. Specifically,
\begin{eqnarray*}
& \wfoo = \{ \sigma \in \Omega^4; \Lambda_{\Ph}(\sigma) = -24 \,
\sigma \} \qquad & \wfot = \{ \sigma \in \Omega^4;
\Lambda_{\Ph}(\sigma) = +4 \, \sigma \} \\ & \wfos = \{ \sigma \in
\Omega^4; \Lambda_{\Ph}(\sigma) = -12 \, \sigma \} \qquad &
\wfoth = \{ \sigma \in \Omega^4; \Lambda_{\Ph}(\sigma) = 0 \}
\end{eqnarray*}
In addition, we have
\begin{equation*}
\wfoo = \{ D(\lambda g); \lambda \in \R \} \qquad \wfos =
\{ D(A_7); A_7 \in \wtws \} \qquad \wfoth = \{ D(A_0); A_0 \in S^2_0 \}
\end{equation*}
where $S^2_0$ is the space of symmetric traceless tensors.
Also, Proposition~\ref{starDAprop} shows that
\begin{equation*}
\Omega^4_+ = \{ \sigma \in \Omega^4; \st \sigma = \sigma \} = \wfoo
\oplus \wfos \oplus \wfot \qquad \qquad \Omega^4_- = \{ \sigma \in
\Omega^4; \st \sigma = - \sigma \} = \wfoth
\end{equation*}
is the decomposition into self-dual and anti-self dual $4$-forms.
\end{prop}
Finally, we have the following result, which is also proved using
Lemma~\ref{SPidentitieslemma}.
\begin{prop} \label{Lambdasquaredprop}
Let $\sigma \in \Omega^4$. Then if we act on $\sigma$ by $\Lambda_{\Ph}$
twice, we have
\begin{eqnarray*}
\Lambda_{\Ph}(\Lambda_{\Ph}(\sigma))_{ijkl} & = &
2 \, \Ph_{ijmn} g^{mp} g^{nq} \sigma_{pqrs} g^{ra} g^{sb} \Ph_{abkl}
+ 2 \, \Ph_{ikmn} g^{mp} g^{nq} \sigma_{pqrs} g^{ra} g^{sb} \Ph_{ablj}
\\ & & {}+ 2 \, \Ph_{ilmn} g^{mp} g^{nq} \sigma_{pqrs} g^{ra} g^{sb}
\Ph_{abjk} + 24 \, \sigma_{ijkl} - 16 \, \Lambda_{\Ph}(\sigma)_{ijkl}
\end{eqnarray*}
\end{prop}

We will need Propositions~\ref{Lambdaprop} and~\ref{Lambdasquaredprop}
in Section~\ref{torsionsec} to study the torsion of a \SPs.

\subsection{The torsion tensor of a \SPs}
\label{torsionsec}

In order to define the {\it torsion tensor\/} $T$ of a \SPs\ $\Ph$,
we need to first study the decomposition of $\nab{X}\Ph$ into
its components in $\Omega^4$.

\begin{lemma} \label{torsionsymmetrieslemma}
For any vector field $X$, the $4$-form $\nab{X} \Ph$ lies in the
subspace $\wfos$ of $\Omega^4$. Hence $\nab{}\Ph$ lies in 
the space $\wone \otimes \wfos$, a $56$-dimensional space (pointwise.)
\end{lemma}
\begin{proof}
Let $X = \ddx{m}$, and consider the $4$-form $\nab{m} \Ph$. Then
the second equation of Proposition~\ref{SPderivativeidentitiesprop}
tells us that $\nab{m} \Ph$ is orthogonal to $S^2 \cong \wfoo \oplus
\wfoth$, exactly as in the \G\ case as discussed in~\cite{K3}.
However, we need to work harder to show that there is no
$\wfot$ component.

The essential reason that $\nab{X}\Ph \in \wfos$ is because
of the way that the $4$-form $\Ph$ determines the metric $g$.
From~\cite{K1} (which uses a different orientation convention), we
have
\begin{equation} \label{SPbasiceq}
(u \hk v \hk \Ph) \wedge (w \hk y \hk \Ph) \wedge \Ph = -6 \, g(u
\wedge v, w \wedge y) \vol + 7 \, \Ph(u,v,w,y) \vol
\end{equation}
Taking $\nab{X}$ of this identity gives
\begin{multline*}
(u \hk v \hk \nab{X}\Ph) \wedge (w \hk y \hk \Ph) \wedge \Ph
+ (u \hk v \hk \Ph) \wedge (w \hk y \hk \nab{X}\Ph) \wedge \Ph \\ {}
+ (u \hk v \hk \Ph) \wedge (w \hk y \hk \Ph) \wedge \nab{X}\Ph =
7 \, \nab{X}\Ph(u,v,w,y) \vol
\end{multline*}
Now since $\st \Ph = \Ph$ and $\st  (\nab{X} \Ph) = \nab{X} \! (\st \Ph) =
\nab{X} \Ph$, this can be written as
\begin{multline*}
g( (u \hk v \hk \nab{X}\Ph) \wedge (w \hk y \hk \Ph) , \Ph) 
+ g( (u \hk v \hk \Ph) \wedge (w \hk y \hk \nab{X}\Ph) , \Ph) 
\\ {}+ g(u \hk v \hk \Ph) \wedge (w \hk y \hk \Ph), \nab{X}\Ph) = 
7 \, \nab{X}\Ph(u,v,w,y)
\end{multline*}
We write this expression in coordinates, use Lemma~\ref{SPidentitieslemma}
to simplify the contractions of $\Ph$ with itself, and skew-symmetrize
the result to obtain
\begin{eqnarray*}
(\nab{X}\Ph)_{ijkl} & = & \frac{1}{12} \left( \Ph_{ijmn} g^{mp} g^{nq}
(\nab{X}\Ph)_{pqrs} g^{ra} g^{sb} \Ph_{abkl}
+ \Ph_{ikmn} g^{mp} g^{nq} (\nab{X}\Ph)_{pqrs} g^{ra} g^{sb} \Ph_{ablj}
\right) \\ & & {} + \frac{1}{12} \left( \Ph_{ilmn} g^{mp} g^{nq}
(\nab{X}\Ph)_{pqrs} g^{ra} g^{sb} \Ph_{abjk}\right) -\frac{1}{3} \left( 
(\nab{X}\Ph \cdot \Ph)_{ijkl} +
(\nab{X}\Ph \cdot \Ph)_{iklj} \right) \\ & & -\frac{1}{3} \left(
(\nab{X}\Ph \cdot \Ph)_{iljk} + (\nab{X}\Ph \cdot \Ph)_{jkil} +
(\nab{X}\Ph \cdot \Ph)_{jlki} + (\nab{X}\Ph \cdot \Ph)_{klij} \right)
\end{eqnarray*}
using the notation of Definition~\ref{Lambdadefn}. In fact the above
expression can also be directly verified using the identities of
Lemma~\ref{SPidentitieslemma} and
Proposition~\ref{SPderivativeidentitiesprop}. Now using
Definition~\ref{Lambdadefn} and Proposition~\ref{Lambdasquaredprop} this
becomes
\begin{equation*}
(\nab{X}\Ph)_{ijkl} = -\frac{1}{3} \, \Lambda_{\Ph}(\nab{X}\Ph)_{ijkl}
+ \frac{1}{24} \left( \Lambda_{\Ph}(\Lambda_{\Ph}(\nab{X}\Ph))_{ijkl}
 - 24 \, (\nab{X}\Ph)_{ijkl} + 16 \, \Lambda_{\Ph}(\nab{X}\Ph)_{ijkl}
\right)
\end{equation*}
Upon simplication, we have finally succeeded in showing that the basic
relation~\eqref{SPbasiceq} between the metric and the \SPs\ $\Ph$ implies
\begin{equation} \label{SPtorsionidentityeq}
\Lambda_{\Ph}(\Lambda_{\Ph}(\nab{X}\Ph)) + 8 \, \Lambda_{\Ph}(\nab{X}\Ph)
- 48 \, (\nab{X} \Ph) = 0 \qquad \qquad \text{for any } X
\end{equation}
Let $\nab{X}\Ph = \sigma_1 + \sigma_7 + \sigma_{27} + \sigma_{35}$
be its decomposition into components, where $\sigma_k \in \Omega^4_k$.
Using Proposition~\ref{Lambdaprop}, equation~\eqref{SPtorsionidentityeq}
says
\begin{equation*}
(336)\, \sigma_1 + (0)\, \sigma_7 + (240)\, \sigma_{27} - (48)\, 
\sigma_{35} = 0
\end{equation*}
which, by linear independence, shows $\sigma_1 = \sigma_{27} = \sigma_{35}
= 0$. Therefore $\nab{X}\Ph \in \wfos$.
\end{proof}
\begin{rmk} \label{torsionsymmetriesrmk1}
The above result was first proved in~\cite{F} by Fern\`andez,
using different methods.
\end{rmk}

\begin{defn} \label{torsiondefn}
Lemma~\ref{torsionsymmetrieslemma} says that $\nab{}\Ph$ can be written as
\begin{equation*}
\nab{m} \Ph_{ijkl} = D(T_m)_{ijkl} = T_{m;ip} g^{pq} \Ph_{qjkl} + T_{m;jp}
g^{pq} \Ph_{iqkl} + T_{m;kp} g^{pq} \Ph_{ijql} + T_{m;lp} g^{pq} \Ph_{ijkq}
\end{equation*}
where for each fixed $m$, $T_{m;ab}$ is in $\wtws$. This defines
the {\it torsion tensor\/} $T$ of the \SPs, which
is an element of $\wone \otimes \wtws$.
\end{defn}

The following lemma gives an explicit formula for $T_{m;ab}$ in terms
of $\nab{}\Ph$. This will be used in Section~\ref{torsionevolutionsec} to
derive the evolution equation for the torsion tensor.

\begin{lemma} \label{torsionlemma}
The torsion tensor $T_{m;\alpha\beta}$ is equal to
\begin{equation} \label{torsioneq}
T_{m;\alpha\beta} = \frac{1}{96} \, (\nab{m} \Ph_{\alpha jkl} )
\Ph_{\beta bcd} g^{jb} g^{kc} g^{ld}
\end{equation}
\end{lemma}
\begin{proof}
This is a simple computation using Definition~\ref{torsiondefn} and
the identities in Lemma~\ref{SPidentitieslemma}.
\end{proof}

We close this section with some remarks about the decomposition of
$T$ into irreducible components. One can show that $\wone \otimes
\wtws \cong \Omega^3 = \wthe \oplus \wthf$. Therefore
the torsion tensor $T$ is actually a $3$-form, with two irreducible
components. In fact under this isomorphism $T$ is essentially $\dl \Ph$,
which is the content of the following result.

\begin{thm}[Fern\'andez, 1986] \label{Fthm}
The \SPs\ corresponding to $\Ph$ is torsion-free if and only if 
$d \Ph = 0$. Since $\st \Ph = \Ph$, this is equivalent to $\dl \Ph = 0$.
\end{thm}

Suppose $M$ is simply-connected, as it must be to admit a metric with
holonomy exactly equal to \SP\ in the compact case (see~\cite{J4}.)
Then as in the \G\ case, which is described in~\cite{K3}, the
component of the torsion in $\wthe$ can always be conformally scaled
away, once we have made the $\wthf$ component vanish, without
changing that other component. Therefore in principle we can restrict
our attention to trying to make the $\wthf$ component of the torsion
vanish, although it is not clear if this is really a
simplification. We will not pursue this here.

\section{General flows of \SPs s} \label{generalflowssec}

In this section we derive the evolution equations for a
general flow $\ddt \Ph$ of a \SPs\ $\Ph$. Let $A_{ij} = h_{ij} + X_{ij}$,
where $h_{ij} \in S^2$ and $X_{ij} \in \wtws$. Then from the discussion
in Section~\ref{formssec}, a general variation of $\Ph$ can be
written as $\ddt \Ph = D(A)$. In coordinates, using~\eqref{DAeq}, this is
\begin{equation} \label{generalfloweq}
\boxed{\, \, \ddt \Ph_{ijkl} \, = \,  A_{im} g^{mn} \Ph_{njkl} +
A_{jm} g^{mn} \Ph_{inkl} + A_{km} g^{mn} \Ph_{ijnl} + A_{lm} g^{mn}
\Ph_{ijkn} \, \, }
\end{equation}

The first thing we need to do is to derive the evolution equations for the
metric $g$ and objects related to the metric, specifically the volume form
$\vol$ and the Christoffel symbols $\Gamma^k_{ij}$. We do this using a
much simpler argument than that presented in~\cite{K3} for the \G\ case.
This method works for that case as well.

\begin{prop} \label{metricevolutionprop}
The evolution of the metric $g_{ij}$ under the
flow~\eqref{generalfloweq} is given by
\begin{equation} \label{metricevolutioneq}
\boxed{\, \, \ddt g_{ij} = 2 \, h_{ij} \, \, }
\end{equation}
\end{prop}
\begin{proof}
We want to know what the first order variation of the metric $g_{\Ph}$ is,
given a first order variation $D(A)$ of the \SPs\ $\Ph$. It suffices to
consider any path $\Ph(t)$ of \SPs s that satisfies
$\left. \ddt \right|_{t=0} \Ph(t) = D(A)$.
We take
\begin{equation*}
\Ph(t) = e^{A t} \cdot \Ph = \frac{1}{24} \, \Ph_{ijkl} \, (e^{At} \dx{i})
\wedge (e^{At} \dx{j}) \wedge (e^{At} \dx{k}) \wedge (e^{At} \dx{l}) 
\end{equation*}
Then if $g = g_{ij} \dx{i} \dx{j}$ is the metric of $\Ph =
\Ph(0)$, it is easy to see that the metric $g(t)$ of $\Ph(t)$ is
\begin{equation*}
g(t) = g_{ij} (e^{At} \dx{i}) (e^{At} \dx{j})
\end{equation*}
Now we differentiate
\begin{eqnarray*}
\left. \ddt \right|_{t=0} g(t) & = & g_{ij} \, (A^i_k \dx{k}) \dx{j} +
g_{ij} \, \dx{i} (A^j_l \dx{l}) = A_{kj} \, \dx{k} \dx{j} + A_{li}
\, \dx{i} \dx{l} \\ & = & (A_{ij} + A_{ji}) \, \dx{i} \dx{j} =
2 \, h_{ij} \, \dx{i} \dx{j}
\end{eqnarray*}
since $h$ is the symmetric part of $A$. This completes the proof.
\end{proof}

\begin{cor} \label{metricevolutioncor}
The evolution of the inverse $g^{ij}$ of the metric, the volume form
$\vol$, and the Christoffel symbols $\Gamma^k_{ij}$, under the
flow~\eqref{generalfloweq}, are given by
\begin{equation*}
\ddt g^{ij} = -2 \, h^{ij} \qquad \qquad \ddt \vol = \tr_g(h)\, \vol 
\qquad \qquad \ddt \Gamma^k_{ij} = g^{kl} \left( \nab{i} h_{jl} +
\nab{j} h_{il} - \nab{l} h_{ij} \right)
\end{equation*}
\end{cor}
\begin{proof}
This is a standard result.
\end{proof}

\subsection{Evolution of the torsion tensor}
\label{torsionevolutionsec}

In this section we derive the evolution equation for the torsion
tensor $T$ of \Ph\ under the general flow~\eqref{generalfloweq}.
We begin with the evolution of $\nab{m} \Ph_{ijkl}$.

\begin{lemma} \label{nabPhevolutionlemma}
The evolution of $\nab{m} \Ph_{ijkl}$ under the
flow~\eqref{generalfloweq} is given by
\begin{eqnarray*}
\ddt \left( \nab{m} \Ph_{ijkl} \right) & = & A_{ip} g^{pq}
(\nab{m} \Ph_{qjkl}) + A_{jp} g^{pq} (\nab{m} \Ph_{iqkl}) + A_{kp} g^{pq}
(\nab{m} \Ph_{ijql}) + A_{lp} g^{pq} (\nab{m} \Ph_{ijkq}) \\
& & {}+ (\nab{p} h_{im})g^{pq} \Ph_{qjkl} + (\nab{p} h_{jm})g^{pq}
\Ph_{iqkl} + (\nab{p} h_{km})g^{pq} \Ph_{ijql} + (\nab{p}h_{lm})g^{pq}
\Ph_{ijkq} \\ & & {}- (\nab{i} h_{pm})g^{pq} \Ph_{qjkl} -
(\nab{j} h_{pm})g^{pq} \Ph_{iqkl} - (\nab{k} h_{pm})g^{pq} \Ph_{ijql} -
(\nab{l}h_{pm})g^{pq} \Ph_{ijkq} \\ & & {}+ (\nab{m} X_{ip})g^{pq}
\Ph_{qjkl} + (\nab{m} X_{jp})g^{pq} \Ph_{iqkl} + (\nab{m} X_{kp})g^{pq}
\Ph_{ijql} + (\nab{m}X_{lp})g^{pq} \Ph_{ijkq} 
\end{eqnarray*}
where $A_{ij} = h_{ij} + X_{ij} \in S^2 \oplus \wtws$.
\end{lemma}
\begin{proof}
Recall that
\begin{equation*}
\nab{m} \Ph_{ijkl} = \ddx{m} \Ph_{ijkl} - \Gamma^n_{mi} \Ph_{njkl}
- \Gamma^n_{mj} \Ph_{inkl} - \Gamma^n_{mk} \Ph_{ijnl} - \Gamma^n_{ml}
\Ph_{ijkn}
\end{equation*}
so if we differentiate this equation with respect to $t$ and simplify,
we obtain
\begin{eqnarray*}
\ddt \left( \nab{m} \Ph_{ijkl} \right) & = & \nab{m} \left( \ddt \Ph_{ijkl}
\right)  - \left(\ddt \Gamma^n_{mi}\right) \Ph_{njkl} \\ & & {}- \left( \ddt
\Gamma^n_{mj} \right) \Ph_{inkl} - \left( \ddt \Gamma^n_{mk} \right)
\Ph_{ijnl} - \left( \ddt \Gamma^n_{ml} \right) \Ph_{ijkn}
\end{eqnarray*}
Now we substitute~\eqref{generalfloweq} and use
Corollary~\ref{metricevolutioncor}. After we use the product rule on the
first term, all the terms involving $\nab{m} h$ cancel in pairs. The
result now follows.
\end{proof}

\begin{thm} \label{torsionevolutionthm}
The evolution of the torsion tensor $T_{m;\alpha\beta}$ under the
flow~\eqref{generalfloweq} is given by
\begin{equation} \label{torsionevolutioneq}
\boxed{\, \, \ddt T_{m;\alpha\beta} = A_{\alpha p} g^{pq} T_{m;q\beta}
- A_{\beta p} g^{pq} T_{m;q\alpha} + \pi_7( \nab{\beta} h_{\alpha m}
- \nab{\alpha} h_{\beta m} + \nab{m} X_{\alpha\beta} ) \, \, }
\end{equation}
where $A_{ij} = h_{ij} + X_{ij}$ is the element of $S^2 \oplus \wtws$
corresponding to the flow of $\Ph$, and $\pi_7$ denotes the projection
onto $\wtws$ of the tensor skew-symmetric in $\alpha, \beta$ for
fixed $m$.
\end{thm}
\begin{proof}
This is a long computation, but is similar in spirit to the analogous
result for \Gs s in~\cite{K3}. We will describe the main steps, and
leave the details to the reader. Begin with Lemma~\ref{torsionlemma},
and differentiate to obtain:
\begin{eqnarray} \nonumber
\ddt T_{m;\alpha\beta} & = & \frac{1}{96} \left( \ddt \nab{m}
\Ph_{\alpha jkl} \right) \Ph_{\beta bcd} g^{jb} g^{kc} g^{ld} +
\frac{1}{96} (\nab{m} \Ph_{\alpha jkl}) \left( \ddt
\Ph_{\beta bcd} \right) g^{jb} g^{kc} g^{ld} \\
\label{torsionevolutiontempeq} & & {} - \frac{6}{96} (\nab{m}
\Ph_{\alpha jkl}) \Ph_{\beta bcd} h^{jb} g^{kc} g^{ld}
\end{eqnarray}
where we have used $\ddt g^{ij} = - 2 \, h^{ij}$ from
Corollary~\ref{metricevolutioncor}. Recall that for a tensor $B_{ij}$
we defined
\begin{equation*}
D(B)_{ijkl} = B_{ip} g^{pq} \Ph_{qjkl} + B_{jp} g^{pq} \Ph_{iqkl}
+ B_{kp} g^{pq} \Ph_{ijql} + B_{lp} g^{pq} \Ph_{ijkq}
\end{equation*}
Let us define a similar shorthand notation $D_m (B)$ to denote
\begin{equation*}
D_m(B)_{ijkl} = B_{ip} g^{pq} \nab{m} \Ph_{qjkl} + B_{jp} g^{pq}
\nab{m} \Ph_{iqkl} + B_{kp} g^{pq} \nab{m} \Ph_{ijql} + B_{lp} g^{pq}
\nab{m} \Ph_{ijkq}
\end{equation*}
Then Lemma~\ref{nabPhevolutionlemma} says that
\begin{equation} \label{nabPhievolutionshorteq}
\ddt \nab{m} \Ph_{ijkl} = D_m ( A ) + D ( B )
\end{equation}
where we define
\begin{equation} \label{Bdefneq}
B_{\alpha \beta} = \nab{\beta} h_{\alpha m} - \nab{\alpha} 
h_{\beta m} + \nab{m} X_{\alpha\beta}
\end{equation}
We also have $\ddt \Ph = D(A)$. Therefore~\eqref{torsionevolutiontempeq}
becomes
\begin{eqnarray} \nonumber
\ddt T_{m;\alpha\beta} & = & \frac{1}{96} D_m (A)_{\alpha jkl}
\Ph_{\beta bcd} g^{jb} g^{kc} g^{ld} + \frac{1}{96} D (B)_{\alpha jkl}
\Ph_{\beta bcd} g^{jb} g^{kc} g^{ld} \\ \label{torsionevolutiontempeq2}
& & {} + \frac{1}{96} (\nab{m} \Ph_{\alpha jkl}) D(A)_{\beta bcd}
g^{jb} g^{kc} g^{ld} - \frac{6}{96} (\nab{m}
\Ph_{\alpha jkl}) \Ph_{\beta bcd} h^{jb} g^{kc} g^{ld}
\end{eqnarray}
We will break up the computation into several manageable pieces.
First, we need the following identity. If $A = h + X \in S^2 \oplus
\wtws$, then:
\begin{eqnarray} \nonumber
D(A)_{ijkl} \Ph_{abcd} g^{kc} g^{ld} & = & 4 \, h_{ia} g_{jb}
- 4 \, h_{ib} g_{ja} + 4 \, h_{jb} g_{ia} - 4 \, h_{ja} g_{ib} +
2 \tr_g(h) (g_{ia} g_{jb} - g_{ib} g_{ja} - \Ph_{ijab} )\\ \nonumber 
{} & & + 16 \, X_{ia} g_{jb} - 16 \, X_{ib} g_{ja} + 16 \, X_{jb} g_{ia}
- 16 \, X_{ja} g_{ib} \\ & & {} - 2 \, A_{ip} g^{pq} \Ph_{qjab}
+ 2 \, A_{jp} g^{pq} \Ph_{qiab} + 2 \, A_{pa} g^{pq} \Ph_{qbij}
- 2 \, A_{pb} g^{pq} \Ph_{qaij} \label{tempidentityeq}
\end{eqnarray}
which can be proved using Lemma~\ref{SPidentitieslemma}. Also, it
is easy to check that if $X$ and $Y$ are both in $\wtws$, then
\begin{equation} \label{tempidentityeq2}
X_{ip} g^{pq} Y_{qj} g^{ia} g^{jb} \Ph_{abkl} = X_{kp} g^{pq} Y_{ql}
- Y_{kp} g^{pq} X_{ql}
\end{equation}
which essentially says that the Lie bracket of two elements of
$\wtws$ is always in $\wtwt$. Now using
the identities~\eqref{tempidentityeq} and~\eqref{tempidentityeq2}, and
Lemma~\ref{SPidentitieslemma} again, along with some patience, one can
establish the following four expressions:
\begin{eqnarray*}
D(A)_{\alpha jkl} \Ph_{\beta bcd} g^{jb} g^{kc} g^{ld} & = & 
24 \, h_{\alpha\beta} + 18 \tr_g(h) \, g_{\alpha\beta} + 
96 \,X_{\alpha\beta} \\ D_m(A)_{\alpha jkl} \Ph_{\beta bcd}
g^{jb} g^{kc} g^{ld} & = & 48 \left( h_{\alpha p} g^{pq} T_{m;q\beta}
+ h_{\beta p} g^{pq} T_{m;q\alpha} + \tr_g(h) \, T_{m;\alpha\beta}
- g(X, T_m) g_{\alpha\beta} \right) \\ (\nab{m}\Ph_{\alpha jkl})
D(A)_{\beta bcd} g^{jb} g^{kc} g^{ld} & = & 48 \left( -h_{\alpha p}
g^{pq} T_{m;q\beta} - h_{\beta p} g^{pq} T_{m;q\alpha} + \tr_g(h) \,
T_{m;\alpha\beta} + g(X, T_m) g_{\alpha\beta} \right) \\ & & {} + 96
\left( X_{\alpha p} g^{pq} T_{m;q\beta} - X_{\beta p} g^{pq}
T_{m;q\alpha} \right) \\ (\nab{m} \Ph_{\alpha jkl}) \Ph_{\beta bcd}
h^{jb} g^{kc} g^{ld} & = & 16 \left( -h_{\alpha p}
g^{pq} T_{m;q\beta} - h_{\beta p} g^{pq} T_{m;q\alpha} + \tr_g(h) \,
T_{m;\alpha\beta} \right)
\end{eqnarray*}
Now we use the above four expressions to simplify
equation~\eqref{torsionevolutiontempeq2}. We need to substitute
$B$ as defined in~\eqref{Bdefneq} for $A$ when we use the first of these
expressions. After much cancellation and collecting like terms, we are left
with exactly~\eqref{torsionevolutioneq}. 
\end{proof}

We remark that, just as in the \G\ case, the terms with $\nab{}h$ and
with $\nab{}X$ play quite different roles in the evolution of the
torsion tensor in equation~\eqref{torsionevolutioneq}. One hopes
that it is possible to choose $h$ and $X$ in terms of $T$ and possibly
also $\nab{}T$ so that the evolution equations have nice properties.
In particular we would like the equation to be parabolic transverse
to the action of the diffeomorphism group, for short-time existence.
Ideally such a flow exists where the $\mathrm{L}^2$-norm $||T||$
of the torsion decreases. These are questions for future research.

\section{Bianchi-type identity and curvature formulas} \label{bianchisec}

In this section, we apply the evolution
equation~\eqref{torsionevolutioneq} to derive a Bianchi-type identity
for manifolds with \SPs. This yields explicit formulas for the Ricci
tensor and part of the Riemann curvature tensor in terms of the
torsion tensor. As the calculations here are extremely similar to
those in~\cite{K3}, we will be brief.

\begin{prop} \label{firstbianchiprop}
The diffeomorphism invariance of the metric $g$ as a function
of the $4$-form $\Ph$ is equivalent to the vanishing of the
$\wfoo \oplus \wfoth$ component of $\nab{Y}\Ph$ for
any vector field $Y$. This is the fact which was proved earlier
in Lemma~\ref{torsionsymmetrieslemma}.
\end{prop}
\begin{proof}
The proof is identical to the \G\ case. In both cases it is due to the
fact that the evolution of the metric $g$ depends only on the symmetric
part $h$ of $A = h + X$. Notice that in the \SP\ case, there is a stronger
result that the $\wfot$ component of $\nab{Y}\Ph$ also vanishes,
Lemma~\ref{torsionsymmetrieslemma}, which does not follow from here.
\end{proof}

\begin{thm} \label{secondbianchithm}
The diffeomorphism invariance of the torsion tensor $T$ as a function
of the $4$-form $\Ph$ is equivalent to the following identity:
\begin{equation} \label{secondbianchieq}
\boxed{\, \, \nab{q} T_{p;\alpha\beta} - \nab{p} T_{q;\alpha\beta} =
\frac{1}{4} R_{pq\alpha\beta} - \frac{1}{8} R_{pqij} g^{ia} g^{jb}
\Ph_{ij\alpha\beta} + 2 \, T_{q; \alpha m} g^{mn} T_{p; n \beta}
-  2 \, T_{p; \alpha m} g^{mn} T_{q; n \beta} \, \, }
\end{equation}
\end{thm}
\begin{proof}
The proof is very similar to the analogous result for \Gs s described
in~\cite{K3}, and is left to the reader. The
identity~\eqref{secondbianchieq} can also be established directly by
using~\eqref{torsioneq}, Lemma~\ref{SPidentitieslemma}, and
the Ricci identities.
\end{proof}

We now examine some consequences of Theorem~\ref{secondbianchithm}.
For $i$ and $j$ fixed, the Riemann curvature tensor $R_{ijkl}$ is
skew-symmetric in $k$ and $l$. Hence we can use the decomposition
of $\Omega^2$ to write it as
\begin{equation*}
R_{ijkl} = (\pi_7 (\mathrm{Riem}))_{ijkl} + (\pi_{21}
(\mathrm{Riem}))_{ijkl}
\end{equation*}
where by equation~\eqref{wtwsprojeq}, we have
\begin{equation} \label{riem7eq}
(\pi_7 (\mathrm{Riem}))_{ijkl} = \frac{1}{4} R_{ijkl} - \frac{1}{8}
R_{ijab} g^{ap} g^{bq} \Ph_{pqkl}
\end{equation}
Therefore the identity~\eqref{secondbianchieq} says that
\begin{equation} \label{pi7riemanneq}
(\pi_7 (\mathrm{Riem}))_{pq\alpha\beta} = \nab{q} T_{p;\alpha\beta} -
\nab{p} T_{q;\alpha\beta} + 2 \left( T_{p; \alpha m} g^{mn} T_{q; n \beta}
- T_{q; \alpha m} g^{mn} T_{p; n \beta} \right)
\end{equation}

\begin{cor} \label{ambrosesingercor}
If $\Ph$ is torsion-free, then the Riemann curvature tensor
$R_{ijkl} \in S^2 (\Omega^2)$ actually takes values in $S^2(\wtwt)$, where
$\wtwt \cong \mathfrak{so}(7)$, the Lie algebra of $\SP$.
\end{cor}
\begin{proof}
Setting $T = 0$ in~\eqref{pi7riemanneq} shows the for fixed $i$, $j$, we
have $R_{ijkl} \in \wtwt$ as a skew-symmetric tensor in $k$, $l$.
The result now follows from the symmetry $R_{ijkl} = R_{klij}$.
\end{proof}

\begin{rmk} \label{ambrosesingerrmk}
This result is well-known. When $T=0$, the Riemannian holonomy of the
metric $g_{\Ph}$ is contained
in the group $\SP$. By the Ambrose-Singer holonomy theorem, the Riemann
curvature tensor of the metric is thus an element of
$S^2 (\mathfrak{so}(7))$.
\end{rmk}

\begin{lemma} \label{riccilemma}
Let $Q_{ijkl} = R_{ijab} g^{ap} g^{bq} \Ph_{pqkl}$. Then we have $Q_{ijkl}
g^{il} = 0$.
\end{lemma}
\begin{proof}
This is identical to the \G\ case proved in~\cite{K3}.
\end{proof}

From Lemma~\ref{riccilemma} and equation~\eqref{riem7eq}, we see that
the Ricci tensor $R_{jk}$ can be expressed as
\begin{equation} \label{riccitempeq}
R_{jk} = R_{ijkl} g^{il} = 4 \, (\pi_7 (\mathrm{Riem}))_{ijkl} g^{il}
\end{equation}

\begin{prop} \label{ricciprop}
Given a \SPs\ $\Ph$ with torsion tensor $T_{m;\alpha\beta}$, its
associated metric $g$ has Ricci curvature $R_{jk}$ given by
\begin{equation*}
R_{jk} = 4 \, g^{il} \nab{i} T_{j;lk} - 4 \, \nab{j} (g^{il} T_{i;lk} )
+ 8 \, T_{j;mk} T_{i; nl} g^{mn} g^{il} - 8 \, T_{j;ml} T_{i;nk} g^{mn}
g^{il}
\end{equation*}
\end{prop}
\begin{proof}
This follows immediately from equations~\eqref{pi7riemanneq}
and~\eqref{riccitempeq}.
\end{proof}

\begin{cor} \label{ricciflatcor}
The metric of a torsion-free \SPs\ is necessarily Ricci-flat.
This is classical, originally proved by Bonan. Here we see a
direct proof of this fact.
\end{cor}

\begin{rmk} \label{ivanovrmk}
Some formulas relating the Ricci tensor and the torsion on a manifold
with \SPs\ have also been obtained by Ivanov in~\cite{I}.
\end{rmk}

\appendix

\section{Identites in \SP-geometry}
\label{spin7identitiessec}

In this appendix we collect several identities involving the
$4$-form $\Ph$ of a \SPs. They are derived by methods analogous to those
for the \G\ case as explained in~\cite{K3}, so we omit the proofs.

In local coordinates $x^1, x^2, \ldots, x^8$, the 4-form $\Ph$ is
\begin{equation*}
\Ph = \frac{1}{24} \Ph_{ijkl} \, dx^i \wedge dx^j \wedge dx^k
\wedge dx^l
\end{equation*}
where $\Ph_{ijkl}$ is totally skew-symmetric. The metric is given by
$g_{ij} = g( \ddx{i}, \ddx{j})$.

\begin{lemma} \label{SPidentitieslemma}
The following identities hold:
\begin{eqnarray*}
\Ph_{ijkl} \Ph_{abcd} g^{ia} g^{jb} g^{kc} g^{ld} & = & 336 \\
\Ph_{ijkl} \Ph_{abcd} g^{jb} g^{kc} g^{ld} & = & 42 g_{ia} \\
\Ph_{ijkl} \Ph_{abcd} g^{kc} g^{ld} & = & 6 g_{ia} g_{jb} - 6
g_{ib} g_{ja} - 4 \Ph_{ijab} \\ \Ph_{ijkl} \Ph_{abcd} g^{ld} & = &
g_{ia} g_{jb} g_{kc} + g_{ib} g_{jc} g_{ka} + g_{ic}
g_{ja} g_{kb} \\ & & {} - g_{ia} g_{jc} g_{kb} - g_{ib} g_{ja}
g_{kc} - g_{ic} g_{jb} g_{ka} \\ & & {} -g_{ia} \Ph_{jkbc} -
g_{ja} \Ph_{kibc} - g_{ka} \Ph_{ijbc} \\ & &  {} -g_{ib} \Ph_{jkca} -
g_{jb} \Ph_{kica} - g_{kb} \Ph_{ijca} \\ & & {} -g_{ic} \Ph_{jkab} -
g_{jc} \Ph_{kiab} - g_{kc} \Ph_{ijab} 
\end{eqnarray*}
\end{lemma}

\begin{prop} \label{SPderivativeidentitiesprop}
The following identities hold:
\begin{eqnarray*}
( \nab{m} \Ph_{ijkl} ) \Ph_{abcd} g^{ia} g^{jb} g^{kc} g^{ld} & = & 0 \\
(\nab{m} \Ph_{ijkl}) \Ph_{abcd} g^{jb} g^{kc} g^{ld} & = &
- \Ph_{ijkl} (\nab{m} \Ph_{abcd} ) g^{jb} g^{kc} g^{ld} \\
(\nab{m} \Ph_{ijkl} ) \Ph_{abcd} g^{kc} g^{ld} & = & - \Ph_{ijkl} 
(\nab{m} \Ph_{abcd} ) g^{kc} g^{ld} - 4 \, \nab{m} \Ph_{ijab}
\end{eqnarray*}
\end{prop}

\end{document}